\documentclass{amsart}

\title[Schubert Calculus and Uniform Property $\Gamma$]{Schubert Calculus and Uniform Property $\Gamma$}
\author{Andrew S. Toms}
\address{Department of Mathematics, Purdue University, 150 N University St, West Lafayette, IN 47906, USA}
\email{atoms@purdue.edu}

\usepackage{amsmath,amssymb,amsfonts,amsthm,mathrsfs}
\usepackage{xcolor}
\usepackage{tikz}
\usepackage{float}
\usepackage[colorlinks=true,linkcolor=blue,citecolor=blue,urlcolor=blue]{hyperref}
\hypersetup{
  pdftitle={Schubert Calculus and the Quadratic Threshold for Uniform Property Gamma},
  pdfauthor={Andrew S. Toms}
}

\usetikzlibrary{arrows.meta,calc,decorations.pathreplacing}

\providecolor{PurdueGold}{HTML}{CEB888}
\providecolor{PurdueBlack}{HTML}{111111}
\providecolor{DeepBlue}{HTML}{17324D}
\providecolor{AccentRed}{HTML}{8A1538}

\theoremstyle{plain}
\newtheorem{theorem}{Theorem}
\newtheorem{proposition}[theorem]{Proposition}
\newtheorem{lemma}{Lemma}
\theoremstyle{definition}
\newtheorem{definition}{Definition}

\theoremstyle{remark}
\newtheorem{remark}{Remark}

\makeatletter
\let\ams@abstract\abstract
\let\endams@abstract\endabstract
\renewcommand{\abstract}[1]{\ams@abstract#1\endams@abstract}
\def\arxiv@stripmscyear 2020: #1\@nil{#1}
\newcommand{\pacs}[2][]{%
  \subjclass[2020]{\arxiv@stripmscyear#2\@nil}}
\makeatother

\AtBeginDocument{\bibliographystyle{amsplain}}
\begin{document}

\abstract{
We construct a simple, separable, unital, nuclear C$^*$-algebra without uniform
property $\Gamma$.  This identifies a new geometric threshold in the theory of nuclear C$^*$-algebras, {\it quadratic dimension growth}. The Thom--Porteous class of a pair of equal-rank bundles over a Grassmannian forces every bundle map between them to vanish
somewhere, and a Schubert calculus calculation propagates this obstruction
across an inductive system to a simple limit.  In the uniform tracial completion of the limit the surviving
obstruction gives projections with identical values on all designated traces
which are not Murray--von Neumann equivalent. This precludes uniform property $\Gamma$.
The Thom--Porteous class responsible for the obstruction lives in degree
quadratic in the forced rank loss. This places the example exactly at the
quadratic dimension-growth scale, synonymous with linear growth in the
natural \(2\)-norm scale, and identifies that scale as the relevant threshold
for uniform property \(\Gamma\).
}

\maketitle
\section{Introduction}

Property $\Gamma$ for $\mathrm{II}_1$ factors was introduced in the foundational work of Murray and von Neumann \cite{MurrayVonNeumann4}.  It distinguishes the hyperfinite factor $\mathcal{R}$ from the free group factors $L(\mathbb{F}_n)$, $n \geq 2$.  It came to further prominence in the 1970s as a key ingredient in Connes' classification of injective factors, where injectivity in the presence of Property $\Gamma$ yields the McDuff property \cite{Connes:Ann}.  Property $\Gamma$ has a natural incarnation of similar import among tracial C$^*$-algebras known as {\it uniform property $\Gamma$} \cite{CETWW:IM}.  It is motivated by the pursuit of a solution to the Toms-Winter regularity conjecture for simple nuclear C$^*$-algebras, which holds that the properties of finite nuclear dimension, $\mathcal{Z}$-stability, and strict comparison coincide for simple separable nuclear C$^*$-algebras and moreover characterize the natural maximal class for which Elliott-type classification via K-theory and traces can hold (modulo the UCT) \cite{TomsWinter:JFA}, \cite{Winter:IM12}.  Indeed, the Toms-Winter conjecture holds in full in the presence of uniform property $\Gamma$ \cite{CETWW:IM}.  It has therefore been a matter of high interest to determine whether every simple separable nuclear C$^*$-algebra has this property.  This is laid out explicitly as Question XIX of \cite{STW:99problems}, which we resolve negatively here:

Until the late 1990s, all of the stock-in-trade simple separable nuclear C$^*$-algebras were tame in the sense that they ultimately have been found to satisfy the Toms-Winter conjecture and enjoy uniform property $\Gamma$. The first example of a simple nuclear C$^*$-algebra which did not satisfy any of the properties of the regularity conjecture was given by Villadsen in 1997 when he constructed an algebra whose ordered $\mathrm{K}_0$-group failed to be weakly unperforated, precluding $\mathcal{Z}$-stability \cite{Villadsen:JFA}. A second construction in 1999 gave the first examples of such algebras with arbitrary stable rank \cite{Villadsen:JAMS}---again, not $\mathcal{Z}$-stable.  Yet despite their pathologies, these algebras nevertheless have uniform property $\Gamma$, the former by recent work of Vaccaro \cite{Vaccaro:preprint} and the latter by \cite{EvingtonSchafhauser:finitedim}.

\begin{theorem}\label{simplenogamma}
    There is a unital simple separable nuclear non-elementary C$^*$-algebra without uniform property $\Gamma$.
\end{theorem}
\noindent
While this is satisfying, the mere existence of the example of Theorem \ref{simplenogamma} is not the central discovery here.  Rather, it is the identification of the geometric threshold governing the presence
or absence of uniform property \(\Gamma\): {\it quadratic dimension growth}.  We situate this property presently.

Our algebra is an approximately homogeneous (AH) algebra, which has the following general form:
\begin{equation}
 A = \lim (A_i, \phi_i), \ \ \ \   A_i = q_i(C(X_i) \otimes \mathcal{K})q_i,
\end{equation}
where $X_i$ is compact Hausdorff and $q_i \in C(X_i) \otimes \mathcal{K}$ is a projection.  The concrete dimension growth of the system is then obtained by examining when
\begin{equation}\label{dgupper}
\liminf \frac{\dim(X_i)}{f(\mathrm{rank}(q_i))} < \infty
\end{equation}
for an increasing function $f:[0,\infty) \to [0,\infty)$.

If, for instance, (\ref{dgupper}) holds with $f(t) = t^2$, then we say that the dimension growth of $A$ is at most quadratic.  If $f(t) = t$ instead, then the dimension growth of $A$ is at most linear, and so on.  To realize a specific growth rate $f$ for $A$ one then also needs that {\it every} AH presentation for $A$ satisfies
\begin{equation}\label{dflower}
\liminf \frac{\dim(X_i)}{f(\mathrm{rank}(q_i))} >0,
\end{equation}
and this is considerably harder to establish.  The algebras we construct here will satisfy (\ref{dgupper}) for $f(x) = x^2$ and will fail to have uniform property $\Gamma$;  in \cite{KT:subquadratic}, E. Kessinger and the author show that any simple unital AH algebra with a presentation satisfying
\begin{equation}\label{ggrowth}
\liminf \frac{\dim(X_i)}{(\mathrm{rank}(q_i))^2} = 0,
\end{equation}
{\it does} have uniform property $\Gamma$.  The apparent quadratic dimension growth of our example here is therefore essential and we obtain the guiding principle that we consider to be our main result:

\vspace{2mm}
\noindent
\begin{center}
{\bf Quadratic dimension growth is the geometric threshold at which \\ the failure of uniform property $\Gamma$ first occurs. }
\end{center}

\begin{figure}[H]
\centering
\begin{tikzpicture}[scale=0.9, every node/.style={font=\small}]
  \def\h{0.18}
  \coordinate (L) at (0.45,0);
  \coordinate (A) at (2.15,0);
  \coordinate (B) at (4.20,0);
  \coordinate (C) at (6.95,0);
  \coordinate (D) at (8.85,0);
  \coordinate (R) at (11.85,0);

  \draw[thick, {Latex[length=3mm]}-{Latex[length=3mm]}] (0.0,0) -- (12.25,0);

  \fill[DeepBlue!25] ($(L)+(0,-\h)$) rectangle ($(A)+(0,\h)$);
  \fill[DeepBlue!50] ($(A)+(0,-\h)$) rectangle ($(B)+(0,\h)$);
  \fill[PurdueGold!45] ($(B)+(0,-\h)$) rectangle ($(C)+(0,\h)$);
  \fill[PurdueGold!80!black!12] ($(C)+(0,-\h)$) rectangle ($(D)+(0,\h)$);
  \fill[AccentRed!18] ($(D)+(0,-\h)$) rectangle ($(R)+(0,\h)$);

  \draw[thick, PurdueBlack!70] ($(L)+(0,-\h)$) rectangle ($(R)+(0,\h)$);
  \draw[PurdueBlack!70] ($(A)+(0,-\h)$) -- ($(A)+(0,\h)$);
  \draw[PurdueBlack!70] ($(B)+(0,-\h)$) -- ($(B)+(0,\h)$);
  \draw[PurdueBlack!70] ($(C)+(0,-\h)$) -- ($(C)+(0,\h)$);
  \draw[PurdueBlack!70] ($(D)+(0,-\h)$) -- ($(D)+(0,\h)$);

  \node[above=5pt] at ($(L)!0.5!(A)+(0,\h)$) {slow};
  \node[above=5pt] at ($(A)!0.5!(B)+(0,\h)$) {linear};
  \node[above=5pt] at ($(B)!0.5!(C)+(0,\h)$) {superlinear};
  \node[above=5pt] at ($(C)!0.5!(D)+(0,\h)$) {quadratic};
  \node[above=5pt] at ($(D)!0.5!(R)+(0,\h)$) {superquadratic};

  \draw[decorate, decoration={brace, mirror, amplitude=6pt}]
    ($(L)+(0,-\h)$) -- ($(C)+(0,-\h)$)
    node[midway, below=10pt, align=center] {subquadratic/2-norm slow growth \\ uniform $\Gamma$ holds};

  \draw[decorate, decoration={brace, mirror, amplitude=6pt}]
    ($(C)+(0,-\h)$) -- ($(R)+(0,-\h)$)
    node[midway, below=10pt, align=center, text width=3.8cm]
    {quadratic or faster/ \\2-norm linear or faster \\ uniform $\Gamma$ may fail};
\end{tikzpicture}
\caption{A coarse taxonomy of geometric dimension-growth regimes.}
\label{fig:geometric-thresholds}
\end{figure}
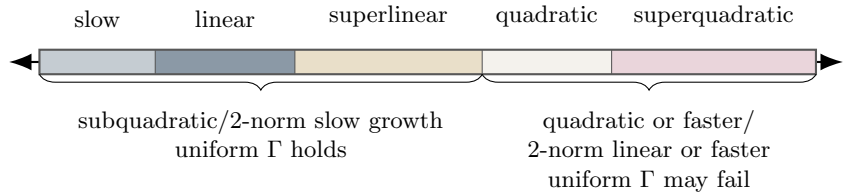

\noindent
Every region of the diagram above is populated by simple separable unital nuclear C$^*$-algebras, and we have drawn an important distinction regarding dimension growth regimes:  the normalized \(2\)-norm on \(M_n\)
is the Hilbert--Schmidt norm divided by \(\sqrt n\), so dimension growth which is quadratic in \(n\) is {\it linear} in the natural \(2\)-norm normalization; all subquadratic growth is hence properly viewed as having the 2-norm analog of the slow dimension growth property (\cite{CGSTW}, \cite{Villadsen:JFA}, \cite{Toms:Ann}, \cite{Toms:JFA}, \cite{Toms:TAMS}, \cite{Toms:Adv}).  Our central conclusion can therefore be recast:  {\it the presence or absence of uniform property $\Gamma$ mirrors the presence or absence of 2-norm slow dimension growth.}  The classical slow dimension growth property pervaded the early days of the classification theory of separable simple nuclear C$^*$-algebras and is known to be equivalent to the properties of the Toms-Winter conjecture for AH algebras.  This at least suggests the presence of a 2-norm version of the Toms-Winter regularity framework in tracially complete C$^*$-algebras.  We defer any finer prognostication for the time being.

The engine of our construction is the Thom--Porteous theory of degeneracy loci, and the obstruction we derive from this theory is propagated across an inductive system through the promised Schubert calculus of the title.  This theory is hardly canon in operator algebras, and so we situate it here relative to better known mechanisms for generating pathology in C$^*$-algebras coming from Villadsen's work and its descendants.

If there is an underlying thesis to Villadsen-type results, it is that unstable phenomena that are present in commutative algebras and their matrix amplifications can often be made to persist in simple nuclear C$^*$-algebras through sufficiently rigid constructions.  Our construction here fits that mold in its broadest outlines, but our methods and obstruction nevertheless represent a sharp break from those that have gone before it.  In Villadsen-style obstructions a desired property is typically translated into an order theoretic property in K-theory or the Cuntz semigroup that is then realized in homogeneous algebras and made durable across an inductive system.  But while these obstructions invariably end with a statement about the failure of comparability of projections that morally should be comparable in light of rank data, our obstruction is instead a careful {\it quantification of failure} in which Villadsen-style arguments can be viewed as a limiting case.
Indeed, suppose $T:E \to F$
is a bundle map from a trivial vector bundle $E$ of rank $k$ into a vector bundle $F$ of rank $l \ge k$ over a finite CW-complex $X$.
If \(T\) is fiberwise injective everywhere, then its image defines a trivial
rank-\(k\) subbundle of $F$. Thus, the classical Villadsen obstruction
to the existence of large trivial subbundles may equivalently be interpreted
as an obstruction to the existence of everywhere-injective bundle maps.

The Thom--Porteous theory generalizes this viewpoint by studying arbitrary
rank-drop conditions for the bundle map $T$. We direct the reader to \cite{FultonIntersectionTheory}, \cite{FultonYoungTableaux} and \cite{ManivelSymmetricFunctions} for a full treatment.  Rather than merely obstructing injectivity, it detects
and quantifies unavoidable rank degeneracy of bundle maps through the Thom--Porteous class
$\Delta_s(F-E)$ of the virtual bundle $F-E$.  We specialize to the case where $\mathrm{rank}(E) = \mathrm{rank}(F) = d$, and our initial model takes $F$ to be the tautological bundle for Grassmannian $\operatorname{Gr}(d,2d)$ and $E$ to be trivial of the same rank.  The nonvanishing of $\Delta_s(F-E)$ entails that at least one fiber map $T_x$ satisfies $\mathrm{rank}(T_x) \leq d-s$, so further specializing to the case $s=d$ forces $T_x = 0$ for some $x \in X$.  This forced vanishing phenomenon is then propagated across an inductive system using Schubert calculus tools including the rectangular resultant formula for supersymmetric Schur functions.  In general
\begin{equation}
    \Delta_d(F-E) \in H^{2d^2}(\operatorname{Gr}(d,2d); \mathbb{Z}),
\end{equation}
so that a non-zero Thom--Porteous class forces the dimension of the base space to be quadratically related to the rank of $F$.  This is the source of our quadratic dimension growth.

Throughout, if \(A\) is a unital \(C^*\)-algebra with nonempty trace simplex
\(T(A)\), we write
\[
   \|a\|_{2,u}=\sup_{\tau\in T(A)} \tau(a^*a)^{1/2},
   \qquad a\in A,
\]
for the uniform tracial \(2\)-seminorm.  The uniform tracial completion \(A^u\) is the quotient of the
\(C^*\)-algebra of operator-norm-bounded
\(\|\cdot\|_{2,u}\)-Cauchy sequences in \(A\) by the ideal of
\(\|\cdot\|_{2,u}\)-null sequences
\cite[Definition~3.19]{CCEGSTW}.
Each \(\tau\in T(A)\) extends uniquely to a uniformly \(2\)-norm continuous
trace on \(A^u\), and these extensions are the designated traces on \(A^u\).
We use the same notation for matrix amplifications, writing \(\tau^{(k)}\)
for the induced trace on \(M_k(A^u)\).

We obstruct uniform property $\Gamma$ by leaning on the fact that a simple unital separable nuclear C$^*$-algebra $A$ with uniform property $\Gamma$ has the property that projections in the uniform tracial completion $A^u$ of $A$ enjoy comparison by designated traces.  In particular, if $P,Q \in A^u$ are projections satisfying $\tau(P) = \tau(Q)$ for every trace on $A$ (here extended to $A^u$), then $P \sim Q$.  Let $V$ be a partial isometry in $A^u$ implementing the equivalence of $P$ and $Q$, so that $V = QVP$.  If $P$ and $Q$ are in the image of $A$ then this last equation can be pulled back to a finite stage of our inductive system approximately in trace, yielding a bundle map $T:E_P \to E_Q$, where $E_P$ and $E_Q$ are bundles corresponding to pre-images of $P$ and $Q$, respectively.  This map is shown to have nonzero rank everywhere on the base space of the finite stage, contradicting the fact that such maps cannot exist in light of our nonvanishing Thom--Porteous class.

We note that by recent work of Vaccaro, any simple unital non-elementary
AH algebra of stable rank one has uniform property \(\Gamma\) \cite{Vaccaro:preprint}. Consequently
the examples constructed here necessarily have stable rank strictly greater
than one, and the Thom--Porteous mechanism developed below provides a new
route to higher stable rank phenomena distinct from the classical constructions
of Villadsen's second type \cite{Villadsen:JAMS}.


We close the introduction by outlining the paper. Section~2 recalls the equal-rank Thom--Porteous obstruction in the form needed below. Section~3 packages this obstruction as a forced-degeneracy criterion and establishes the initial Grassmannian model. Sections~4 and~5 construct the non-simple inductive system and prove, by a Schubert-calculus computation, that the Thom--Porteous obstruction propagates through it. Section~6 shows that the surviving obstruction gives non-equivalent projections with identical designated trace values in the uniform tracial completion, and hence obstructs uniform property $\Gamma$. Section~7 uses a sparse point-evaluation simplefication to obtain the announced simple AH example.

\vspace{2mm}
\noindent
{\bf Acknowledgements:}  The author thanks Prof. Stuart White, the University of Oxford, and the Leverhulme Trust for hosting his tenure of a Leverhulme Trust Visiting Professorship at Oxford in the fall of 2025, and to Prof. White in particular for helpful comments on earlier drafts.  The author thanks ChatGPT for pointing him toward the Thom–Porteous formula and related Schubert-calculus computational tools during the exploratory stage of this project. The mathematical ideas, arguments and proofs in the sequel are due entirely to the author.

\section{Thom--Porteous obstruction in the equal-rank case}

Recall that a projection in a homogeneous \(C^*\)-algebra defines a vector
bundle over the spectrum, and that if \(p,q\in C(X)\otimes\mathcal K\) are
projections then an element of
\[
q(C(X)\otimes\mathcal K)p
\]
is a bundle map from the bundle represented by \(p\) to the bundle represented
by \(q\).

Let \(E,F\to X\) be complex vector bundles of the same rank \(n\) over a
finite CW complex, and let \(T:E\to F\) be a bundle map.  For
\(1\leq s\leq n\), set
\[
D_s(T)=\{x\in X:\dim_{\mathbb C}\ker T_x\geq s\}.
\]
Equivalently, by rank--nullity,
\[
D_s(T)=\{x\in X:\operatorname{rank}(T_x)\leq n-s\}.
\]
This is Fulton's rank-degeneracy locus \(D_{n-s}(T)\); our subscript records
the forced kernel dimension.  The set \(D_s(T)\) is closed, since locally it is
cut out by the vanishing of the \((n-s+1)\times(n-s+1)\) minors of a matrix for
\(T\).  In the extreme case \(s=n\),
\[
D_n(T)=\{x\in X:T_x=0\}.
\]

The associated Thom--Porteous class is
\[
\Delta_s(c(F-E))
:=
\Delta_s^{(s)}(c(F-E))
=
\det(c_{s+j-i}(F-E))_{1\leq i,j\leq s}
=
s_{(s^s)}(F-E),
\]
where \(c(F-E)=c(F)c(E)^{-1}\).  After passing to a splitting space, if the
Chern roots of \(F\) and \(E\) are written as alphabets \(U\) and \(V\),
respectively, this class is the Schur function in a difference of alphabets
\[
s_{(s^s)}(U-V).
\]
This is the Schur-function convention used throughout the Schubert-calculus
calculation in Section~\ref{schubert}.

We use the topological supported form of the Thom--Porteous obstruction.
This point is worth making explicit because the bundle maps arising below from
homogeneous \(C^*\)-algebras are continuous bundle maps, not algebraic morphisms.
For a continuous complex bundle map \(T:E\to F\) over a finite CW complex, the
universal determinantal Thom class pulls back to a class
\[
\widetilde\Delta_s(T)\in H^{2s^2}_{D_s(T)}(X;\mathbb Z)
\]
whose image in ordinary cohomology is
\[
\Delta_s(c(F-E))
=
\det(c_{s+j-i}(F-E))_{1\leq i,j\leq s}.
\]
Thus, if \(D_s(T)=\varnothing\), then
\[
H^{2s^2}_{D_s(T)}(X;\mathbb Z)=0,
\]
and hence
\[
\Delta_s(c(F-E))=0.
\]
Equivalently,
\[
\Delta_s(c(F-E))\neq 0
\quad\Longrightarrow\quad
D_s(T)\neq\varnothing .
\]
See Porteous \cite{PorteousSimpleSingularities} for the topological form and
Fulton \cite[Chapter~14, \S14.4]{FultonIntersectionTheory} or Pragacz
\cite{Pra88} for the intersection-theoretic formulation.

In the extreme case \(s=n\), which is the only case used in the obstruction
below, this can be seen directly without invoking the full Thom--Porteous
machinery. A bundle map \(T:E\to F\) is a section of
\[
\operatorname{Hom}(E,F)\cong F\otimes E^*,
\]
and \(D_n(T)\) is precisely its zero set. If \(D_n(T)=\varnothing\), then
\(F\otimes E^*\) has a nowhere-zero section, so its Euler class, equivalently
its top Chern class, vanishes:
\[
c_{n^2}(F\otimes E^*)=0.
\]
On a splitting space, if \(U=(u_1,\ldots,u_n)\) and \(V=(v_1,\ldots,v_n)\)
are the Chern-root alphabets of \(F\) and \(E\), then
\[
c_{n^2}(F\otimes E^*)
=
\prod_{i,j}(u_i-v_j)
=
s_{(n^n)}(U-V)
=
\Delta_n(c(F-E)).
\]
Thus
\[
\Delta_n(c(F-E))\neq 0
\quad\Longrightarrow\quad
D_n(T)\neq\varnothing,
\]
so every bundle map \(E\to F\) vanishes somewhere.


\section{Forced degeneracy and the initial model}

\subsection{Forced degeneracy of bundle maps}

Let \(X\) be a finite CW complex and let \(E,F\to X\) be complex vector bundles
of the same constant rank \(n\).

\begin{definition}
We say that the pair \((E,F)\) is \emph{\(s\)-degeneracy-forcing} if, for
every bundle map \(T:E\to F\), there exists \(x\in X\) such that
\[
\dim_{\mathbb C}\ker T_x\geq s.
\]
We say that \((E,F)\) is \emph{totally degeneracy-forcing} if it is
\(n\)-degeneracy-forcing, equivalently, if every bundle map \(T:E\to F\)
vanishes at some point of \(X\).
\end{definition}

By the conclusion of Section~2, if
\[
\Delta_s(c(F-E))\neq 0,
\]
then \((E,F)\) is \(s\)-degeneracy-forcing.  In particular, if
\(\Delta_n(c(F-E))\neq 0\), then \((E,F)\) is totally degeneracy-forcing.

\subsection{The initial Grassmannian model}\label{grassman}
Let $X_1=\operatorname{Gr}(d,2d)$, and let $S\to X_1$
be the tautological rank \(d\) bundle. Fix a \(d\)-plane $P\subset \mathbb C^{2d}$
and let $\theta^d=X_1\times P$ be the corresponding trivial rank \(d\) bundle.
We set
\[
q_1=S,
\qquad
p_1=\theta^d.
\]
Thus \(p_1\) and \(q_1\) have the same rank \(d\).
The relevant Thom--Porteous class is
\[
\Delta_d(c(q_1-p_1))
=
\Delta_d(c(S-\theta^d))
\in
H^{2d^2}(\operatorname{Gr}(d,2d);\mathbb Z).
\]
Since \(p_1=\theta^d\) is trivial, this is the Schur class
\(s_{(d^d)}(S)\). By the standard Schubert calculus of the
Grassmannian, equivalently by the Giambelli formula, this is the top
rectangular Schubert class, hence the class of a point up to sign; see
\cite[Chapter 9]{FultonYoungTableaux}. In particular,
\[
\Delta_d(c(q_1-p_1))\neq 0.
\]
Therefore the pair $(p_1,q_1)$ is totally degeneracy-forcing.
Equivalently, every bundle map \mbox{$T:p_1\to q_1$} vanishes somewhere.

\begin{remark}
This is where our quadratic phenomenon first shows up. The obstruction
to forcing a \(d\)-dimensional kernel lives in degree \(2d^2\), which is the
top cohomological degree of \(\operatorname{Gr}(d,2d)\). This initial obstruction can also be realized on $(S^2)^{d^2}$ or on the complex projective space $\mathbb{CP}^{d^2}$.  The Grassmannian model is chosen because it is canonically aligned with Schubert calculus.
\end{remark}

\section{The inductive construction}\label{nonsimpleconstruction}

We now build the non-simple AH limit
\[
A=\varinjlim(A_i,\phi_i),
\qquad
A_i=q_i(C(X_i)\otimes\mathcal K)q_i.
\]
The limit is unital, separable, and nuclear.  It is not simple at this stage,
but the ranks \(\operatorname{rank}(q_i)\) tend to infinity, so it has no
finite-dimensional representations.  Section~6 proves that this non-simple
limit fails uniform property \(\Gamma\); Section~7 then adds sparse
point-evaluation summands to obtain a simple limit while preserving the
Thom--Porteous obstruction in the uniform tracial completion.

Take \(X_1=\operatorname{Gr}(d,2d)\) and \(p_1,q_1\to X_1\) as in
Subsection~\ref{grassman}, and set \(d_1=d\).  Assume that
\(X_i,p_i,q_i\) have been constructed and that
\[
\operatorname{rank}(p_i)
=
\operatorname{rank}(q_i)
=
d_i.
\]
Set
\[
d_{i+1}=2d_i,
\qquad
j_i=2d_i^2,
\]
and define
\[
X_{i+1}=X_i\times X_i\times \mathbb{CP}^{j_i}.
\]
Let \(D_i=\dim(X_i)\), where dimension means covering dimension.  Since
\(\dim(\operatorname{Gr}(d,2d))=2d^2\) and
\(\dim(\mathbb{CP}^{j_i})=2j_i\), we have
\[
D_1=2d_1^2,
\qquad
D_{i+1}=2D_i+2j_i=2D_i+4d_i^2.
\]
As \(d_{i+1}=2d_i\), the ratio \(R_i=D_i/d_i^2\) satisfies
\[
R_{i+1}=\frac{1}{2}R_i+1,
\qquad
R_1=2.
\]
Hence \(R_i=2\) for every \(i\).  Thus the non-simple system has exactly
quadratic dimension growth in this presentation:
\[
\frac{\dim(X_i)}{\operatorname{rank}(q_i)^2}=2.
\]
The simplefication in Section~\ref{sec:simple-example} only enlarges the unit
rank, so the resulting simple system also has at most quadratic dimension growth.
Let
\[
\pi_1,\pi_2:X_{i+1}\to X_i
\]
be the two coordinate projections, and let
\[
\pi_{\mathbb{CP}^{j_i}}:X_{i+1}\to \mathbb{CP}^{j_i}
\]
be the projection onto the projective-space factor.

Choose nontrivial line bundles \(\xi_{i,1}\) and \(\xi_{i,2}\) over
\(\mathbb{CP}^{j_i}\) such that
\[
c_1(\xi_{i,1}\otimes \xi_{i,2}^{-1})^{j_i}\neq 0
\quad
\text{in}
\quad
H^{2j_i}(\mathbb{CP}^{j_i};\mathbb Z).
\]
For example, one may take
\[
\xi_{i,1}=\mathcal O(1),
\qquad
\xi_{i,2}=\mathcal O(2).
\]
Then
\[
\xi_{i,1}\otimes \xi_{i,2}^{-1}
\cong
\mathcal O(-1),
\]
so its first Chern class has nonzero \(j_i\)-th power on
\(\mathbb{CP}^{j_i}\).
Set
\[
\gamma_{i,r}
=
\pi_{\mathbb{CP}^{j_i}}^*(\xi_{i,r}),
\qquad r=1,2.
\]
Define
\[
p_{i+1}
=
(\pi_1^*p_i\otimes \gamma_{i,1})
\oplus
(\pi_2^*p_i\otimes \gamma_{i,2}),
\]
and
\[
q_{i+1}
=
(\pi_1^*q_i\otimes \gamma_{i,1})
\oplus
(\pi_2^*q_i\otimes \gamma_{i,2}).
\]
Then
\[
\operatorname{rank}(p_{i+1})
=
\operatorname{rank}(q_{i+1})
=
2d_i
=
d_{i+1},
\]
and we set
\[
A_{i+1}=q_{i+1}(C(X_{i+1})\otimes\mathcal K)q_{i+1}.
\]
After representing the line bundles \(\gamma_{i,1}\) and \(\gamma_{i,2}\) by
line-bundle projections, the formula
\[
\phi_i(a)
=
(\pi_1^*a\otimes \gamma_{i,1})
\oplus
(\pi_2^*a\otimes \gamma_{i,2})
\]
defines a \(*\)-homomorphism
\[
\phi_i:C(X_i)\otimes\mathcal K\to C(X_{i+1})\otimes\mathcal K.
\]
Since \(\phi_i(q_i)=q_{i+1}\), it restricts to a unital
\(*\)-homomorphism \(\phi_i:A_i\to A_{i+1}\).

For the later comparison argument, we also regard the bundles \(p_i\) as
projections in a fixed matrix amplification.  Choose \(k\) large enough that
the trivial rank-\(d\) bundle \(p_1=\theta^d\) embeds as a subbundle of
\(q_1^{\oplus k}\).  We then regard \(p_1\) as a projection in \(M_k(A_1)\)
and define
\[
p_i=\phi^{(k)}_{i,1}(p_1)\in M_k(A_i),
\]
where
\[
\phi^{(k)}_{i,1}:M_k(A_1)\longrightarrow M_k(A_i)
\]
is the amplified connecting map.  This projection represents the recursively
defined bundle \(p_i\) over \(X_i\), and we use the same notation for the
bundle and for its representing projection.
\begin{remark}
The two operations in the connecting map play different roles. The product
\(X_i\times X_i\) doubles the old obstruction, a technique introduced by Villadsen in \cite{Villadsen:JFA}.  The projective-space
factor and the two nontrivial, non-isomorphic line bundles supply the extra
cohomological degree needed for the next square Thom--Porteous class to be
nonzero, and they mirror Villadsen's techniques in \cite{Villadsen:JAMS}.  Our connecting maps are therefore a blend of Villadsen maps of the first and second type.
\end{remark}
\section{Propagation of forced degeneracy}\label{schubert}
We now prove that the pairs constructed above remain totally degeneracy-forcing at every
stage. We need to show that the relevant Thom--Porteous class remains
nonzero after passing from \(X_i\) to
\[
X_{i+1}=X_i\times X_i\times \mathbb{CP}^{j_i}.
\]
This is the Schubert calculus core of the construction. We give a significantly detailed account as these techniques are generally not familiar to operator algebraists.
\begin{lemma}
Suppose
\[
\Delta_{d_i}(c(q_i-p_i))\neq 0
\quad\text{in}\quad
H^{2d_i^2}(X_i;\mathbb Z).
\]
Then
\[
\Delta_{2d_i}(c(q_{i+1}-p_{i+1}))\neq 0
\quad\text{in}\quad
H^{2(2d_i)^2}(X_{i+1};\mathbb Z).
\]
\end{lemma}
\begin{proof}
Recall that
\[
X_{i+1}=X_i\times X_i\times \mathbb{CP}^{j_i},
\qquad
j_i=2d_i^2,
\]
and
\[
p_{i+1}
=
(\pi_1^*p_i\otimes \gamma_{i,1})
\oplus
(\pi_2^*p_i\otimes \gamma_{i,2}),
\]
while
\[
q_{i+1}
=
(\pi_1^*q_i\otimes \gamma_{i,1})
\oplus
(\pi_2^*q_i\otimes \gamma_{i,2}).
\]
Thus, as a virtual bundle,
\[
q_{i+1}-p_{i+1}
=
(\pi_1^*(q_i-p_i)\otimes \gamma_{i,1})
\oplus
(\pi_2^*(q_i-p_i)\otimes \gamma_{i,2}).
\]
Set
\[
y_i=c_1(\gamma_{i,1}\otimes \gamma_{i,2}^{-1})
\in H^2(X_{i+1};\mathbb Z).
\]
By construction,
\[
y_i^{j_i}=y_i^{2d_i^2}\neq 0.
\]
We must prove that the square Thom--Porteous class
\[
\Delta_{2d_i}(c(q_{i+1}-p_{i+1}))
\]
is nonzero.
\medskip

\noindent
{\bf Step 1: Passing to a splitting space.}
We use the splitting principle. There is a space
\[
\rho:\widetilde X_{i+1}\longrightarrow X_{i+1}
\]
such that the pullbacks of the bundles appearing above split into sums of line
bundles, and such that the pullback map
\[
\rho^*:H^*(X_{i+1};\mathbb Z)
\longrightarrow
H^*(\widetilde X_{i+1};\mathbb Z)
\]
is injective. The line-bundle twists \(\gamma_{i,1}\) and \(\gamma_{i,2}\) are already
line bundles pulled back from the \(\mathbb{CP}^{j_i}\)-factor, so no splitting
is required for them. In particular,
\[
\rho^*y_i=\rho^*c_1(\gamma_{i,1}\otimes\gamma_{i,2}^{-1})
\]
is the pullback of the class \(y_i\) whose K\"unneth degree on \(X_{i+1}\) is
entirely in the projective-space factor. (See, for example, Milnor--Stasheff~\cite[Chapter 14]{MilnorStasheffCharacteristicClasses} for the splitting principle in this form.)
Therefore it is enough to show that
\[
\rho^*\Delta_{2d_i}(c(q_{i+1}-p_{i+1}))\neq 0.
\]
Indeed, if a cohomology class becomes nonzero after applying an injective map, then it
was already nonzero before applying that map.
We shall use the splitting space only to compute the relevant Schur polynomial by
Chern roots. The important point is that \(q_i-p_i\) is a virtual bundle. Thus its
Chern-root notation must be understood as a difference of two alphabets, one coming
from \(q_i\) and one coming from \(p_i\).
\medskip

\noindent
{\bf Step 2: Factorization of the square Schur class.}
We now compute the part of
\[
\rho^*\Delta_{2d_i}(c(q_{i+1}-p_{i+1}))
\]
with maximal degree in the class
\[
a-b
=
c_1(\rho^*\gamma_{i,1})-c_1(\rho^*\gamma_{i,2}).
\]
Let
\[
A=(\alpha_1,\ldots,\alpha_{d_i})
\]
be the Chern roots of \(\rho^*\pi_1^*q_i\), and let
\[
B=(\beta_1,\ldots,\beta_{d_i})
\]
be the Chern roots of \(\rho^*\pi_1^*p_i\). Similarly, let
\[
C=(\gamma_1,\ldots,\gamma_{d_i})
\]
be the Chern roots of \(\rho^*\pi_2^*q_i\), and let
\[
D=(\delta_1,\ldots,\delta_{d_i})
\]
be the Chern roots of \(\rho^*\pi_2^*p_i\).
Thus the virtual bundle
\[
\rho^*\pi_1^*(q_i-p_i)
\]
is represented, under the splitting principle, by the difference of alphabets
\[
A-B,
\]
and the virtual bundle
\[
\rho^*\pi_2^*(q_i-p_i)
\]
is represented by the difference of alphabets
\[
C-D.
\]
Write
\[
a=c_1(\rho^*\gamma_{i,1}),
\qquad
b=c_1(\rho^*\gamma_{i,2}).
\]
Tensoring by a line bundle shifts every Chern root by the first Chern class of that
line bundle. Therefore
\[
\rho^*(\pi_1^*q_i\otimes\gamma_{i,1})
\]
has Chern roots
\[
A+a=(\alpha_1+a,\ldots,\alpha_{d_i}+a),
\]
while
\[
\rho^*(\pi_1^*p_i\otimes\gamma_{i,1})
\]
has Chern roots
\[
B+a=(\beta_1+a,\ldots,\beta_{d_i}+a).
\]
Similarly,
\[
\rho^*(\pi_2^*q_i\otimes\gamma_{i,2})
\]
has Chern roots
\[
C+b=(\gamma_1+b,\ldots,\gamma_{d_i}+b),
\]
and
\[
\rho^*(\pi_2^*p_i\otimes\gamma_{i,2})
\]
has Chern roots
\[
D+b=(\delta_1+b,\ldots,\delta_{d_i}+b).
\]
It follows that the virtual bundle
\[
\rho^*(q_{i+1}-p_{i+1})
\]
is represented by the difference of alphabets
\[
\big((A+a)\cup(C+b)\big)
-
\big((B+a)\cup(D+b)\big).
\]
We use the following convention for Schur functions in a difference of alphabets. If
\mbox{\(\lambda=(\lambda_1,\ldots,\lambda_\ell)\)} is a partition and \(F-E\) is a virtual
bundle, then
\[
s_\lambda(F-E)
=
\det\big(c_{\lambda_r+s-r}(F-E)\big)_{1\leq r,s\leq \ell}.
\]
In particular, for the square partition
\[
\lambda = \bigl((2d_i)^{2d_i}\bigr)
=
(
\underbrace{2d_i,\,2d_i,\,\ldots,\,2d_i}_{2d_i\ \mathrm{times}}
).
\]
this determinant is
\[
\det\big(c_{2d_i+s-r}(q_{i+1}-p_{i+1})\big)_{1\leq r,s\leq 2d_i},
\]
which is precisely the Thom--Porteous class
\[
\Delta_{2d_i}(c(q_{i+1}-p_{i+1})).
\]
Thus
\[
\rho^*\Delta_{2d_i}(c(q_{i+1}-p_{i+1}))
=
s_{((2d_i)^{2d_i})}
\big((A+a,C+b)-(B+a,D+b)\big).
\]
This is a supersymmetric Schur function in a difference of alphabets.
Indeed, the Thom--Porteous class is a Schur class of the virtual bundle
\(q_{i+1}-p_{i+1}\), not of an honest bundle alone. After passing to the
splitting space, write the Chern roots of \(q_{i+1}\) as the alphabet
\[
U=(A+a)\cup(C+b),
\]
and the Chern roots of \(p_{i+1}\) as the alphabet
\[
V=(B+a)\cup(D+b).
\]
The total Chern class of the virtual bundle is then
\[
c(q_{i+1}-p_{i+1})
=
\frac{\prod_{\xi \in U}(1+\xi)}{\prod_{\nu \in V}(1+\nu)}.
\]
Consequently the Jacobi--Trudi determinant
\[
\det\big(c_{2d_i+k-j}(q_{i+1}-p_{i+1})\big)_{1\leq j,k\leq 2d_i}
\]
is precisely the supersymmetric Schur function
\[
s_{((2d_i)^{2d_i})}(U-V).
\]
Since both \(U\) and \(V\) have \(2d_i\) elements, the rectangular
resultant identity for Schur functions in a difference of alphabets gives
\[
s_{((2d_i)^{2d_i})}(U-V)
=
\prod_{\xi \in U}\prod_{\nu \in V}(\xi-\nu).
\]
This is the standard Jacobi--Trudi description of supersymmetric Schur functions,
or Schur functions in a difference of alphabets; see Macdonald~\cite[Chapter I, \S3, Example 9]{MacdonaldSymmetricFunctions} or \cite[Lemma~1.2, p.~418]{Pra88}.
Therefore
\[
\begin{aligned}
&s_{((2d_i)^{2d_i})}
\big((A+a,C+b)-(B+a,D+b)\big) \\
&\quad =
\prod_{\xi \in (A+a,C+b)}
\prod_{\nu \in (B+a,D+b)}
(\xi-\nu).
\end{aligned}
\]
Expanding this product according to the four possible pairs of blocks gives
\[
\begin{aligned}
&\prod_{\alpha\in A}\prod_{\beta\in B}
\big((\alpha+a)-(\beta+a)\big) \\
&\quad\cdot
\prod_{\alpha\in A}\prod_{\delta\in D}
\big((\alpha+a)-(\delta+b)\big) \\
&\quad\cdot
\prod_{\gamma\in C}\prod_{\beta\in B}
\big((\gamma+b)-(\beta+a)\big) \\
&\quad\cdot
\prod_{\gamma\in C}\prod_{\delta\in D}
\big((\gamma+b)-(\delta+b)\big).
\end{aligned}
\]
The first and fourth factors do not involve \(a-b\). Indeed,
\[
(\alpha+a)-(\beta+a)=\alpha-\beta
\]
and
\[
(\gamma+b)-(\delta+b)=\gamma-\delta.
\]
Thus
\[
\prod_{\alpha\in A}\prod_{\beta\in B}
\big((\alpha+a)-(\beta+a)\big)
=
\prod_{\alpha\in A}\prod_{\beta\in B}
(\alpha-\beta),
\]
and
\[
\prod_{\gamma\in C}\prod_{\delta\in D}
\big((\gamma+b)-(\delta+b)\big)
=
\prod_{\gamma\in C}\prod_{\delta\in D}
(\gamma-\delta).
\]
By the same rectangular resultant identity, these two products are the smaller
square Schur classes:
\[
\prod_{\alpha\in A}\prod_{\beta\in B}
(\alpha-\beta)
=
s_{(d_i^{d_i})}(A-B),
\]
and
\[
\prod_{\gamma\in C}\prod_{\delta\in D}
(\gamma-\delta)
=
s_{(d_i^{d_i})}(C-D).
\]
Equivalently,
\[
s_{(d_i^{d_i})}(A-B)
=
\rho^*\pi_1^*\Delta_{d_i}(c(q_i-p_i)),
\]
and
\[
s_{(d_i^{d_i})}(C-D)
=
\rho^*\pi_2^*\Delta_{d_i}(c(q_i-p_i)).
\]
It remains to examine the two cross factors, since these are the only factors involving
\(a-b\). Set
\[
u=a-b.
\]
Then
\[
(\alpha+a)-(\delta+b)
=
(\alpha-\delta)+u,
\]
whereas
\[
(\gamma+b)-(\beta+a)
=
(\gamma-\beta)-u.
\]
Hence the product of the two cross factors is
\[
\prod_{\alpha\in A}\prod_{\delta\in D}
\big((\alpha-\delta)+u\big)
\cdot
\prod_{\gamma\in C}\prod_{\beta\in B}
\big((\gamma-\beta)-u\big).
\]
There are \(d_i^2\) factors in the first product and \(d_i^2\) factors in the second
product. Therefore the largest possible total degree in \(u=a-b\) is
\[
2d_i^2.
\]
Moreover, the contribution of degree \(2d_i^2\) is obtained in exactly one way:
from every factor in the first cross product one must select the \(u\)-term, and
from every factor in the second cross product one must select the \(-u\)-term. Any
selection of a term \(\alpha-\delta\) or \(\gamma-\beta\) from even one cross factor
lowers the total \(u\)-degree.
Thus the coefficient of \(u^{2d_i^2}\) in the product of the two cross factors is
\[
(-1)^{d_i^2}.
\]
At this point it is important that the expansion is being organized as a
polynomial in the single class
\[
u=a-b=\rho^*y_i,
\]
and not as a polynomial separately in \(a\) and \(b\).  This distinction matters
because, after the later choice of nontrivial line bundles
\(\xi_{i,1}\) and \(\xi_{i,2}\), both \(a\) and \(b\) may have nonzero
projective-space components.  Nevertheless the cross factors depend on these
line-bundle twists only through their difference \(a-b\):
\[
(\alpha+a)-(\delta+b)=(\alpha-\delta)+u,
\qquad
(\gamma+b)-(\beta+a)=(\gamma-\beta)-u.
\]
The same-block factors contain no \(a\) or \(b\) at all, since
\[
(\alpha+a)-(\beta+a)=\alpha-\beta,
\qquad
(\gamma+b)-(\delta+b)=\gamma-\delta.
\]
Thus the entire expression is a polynomial in \(u\), with coefficients pulled
back from the two \(X_i\)-factors.  The unique term of maximal \(u\)-degree is
obtained by selecting the \(u\)-term from every first cross factor and the
\(-u\)-term from every second cross factor.  All other choices have strictly
smaller \(u\)-degree.

Equivalently, possible monomials involving \(a\) and \(b\) separately do not
represent additional top-degree contributions on the projective-space factor;
they occur only after expanding powers of the single class \(u=a-b\).  Since
\(u=\rho^*y_i\), the distinguished term is the unique contribution with
projective-space degree \(4d_i^2\), while every other term has strictly smaller
projective-space degree.
Combining this with the two non-cross factors computed
above, we obtain
\[
\begin{aligned}
\big[u^{2d_i^2}\big]\,
\rho^*\Delta_{2d_i}(c(q_{i+1}-p_{i+1}))
&=
(-1)^{d_i^2}
s_{(d_i^{d_i})}(A-B)
s_{(d_i^{d_i})}(C-D),
\end{aligned}
\]
where \(\big[u^{2d_i^2}\big]\) denotes the coefficient of \(u^{2d_i^2}\).
Substituting back \(u=a-b\), this says that
\[
\rho^*\Delta_{2d_i}(c(q_{i+1}-p_{i+1}))
\]
contains the distinguished maximal \((a-b)\)-degree term
\[
(-1)^{d_i^2}
s_{(d_i^{d_i})}(A-B)
s_{(d_i^{d_i})}(C-D)
(a-b)^{2d_i^2}.
\]
Equivalently, this distinguished term is
\[
\begin{aligned}
(-1)^{d_i^2}
\rho^*\pi_1^*\Delta_{d_i}(c(q_i-p_i))
\cdot
\rho^*\pi_2^*\Delta_{d_i}(c(q_i-p_i))
\cdot
(a-b)^{2d_i^2}.
\end{aligned}
\]
Finally,
\[
\begin{aligned}
a-b
&=
c_1(\rho^*\gamma_{i,1})-c_1(\rho^*\gamma_{i,2}) \\
&=
c_1(\rho^*(\gamma_{i,1}\otimes \gamma_{i,2}^{-1})) \\
&=
\rho^*c_1(\gamma_{i,1}\otimes \gamma_{i,2}^{-1}) \\
&=
\rho^*y_i.
\end{aligned}
\]
Therefore the distinguished maximal-degree term may be written as
\[
\begin{aligned}
(-1)^{d_i^2}
\rho^*\pi_1^*\Delta_{d_i}(c(q_i-p_i))
\cdot
\rho^*\pi_2^*\Delta_{d_i}(c(q_i-p_i))
\cdot
\rho^*(y_i^{2d_i^2}).
\end{aligned}
\]
The sign $(-1)^{d_i^2}$ will play no role below; the important point is that the coefficient is nonzero. We now record explicitly why this distinguished term cannot be cancelled by any other term in the expansion.
Recall that
\[
u=a-b=\rho^*y_i,
\]
where
\[
y_i=c_1(\gamma_{i,1}\otimes \gamma_{i,2}^{-1})
\]
is pulled back from the \(\mathbb{CP}^{j_i}\)-factor of
\[
X_{i+1}=X_i\times X_i\times \mathbb{CP}^{j_i}.
\]
Since \(j_i=2d_i^2\), the class \(y_i^{2d_i^2}\) is the pullback of the top-degree class
\[
c_1(\xi_{i,1}\otimes \xi_{i,2}^{-1})^{2d_i^2}
\in H^{4d_i^2}(\mathbb{CP}^{2d_i^2};\mathbb Z).
\]
Thus, under the K\"unneth decomposition, \(y_i^{2d_i^2}\) has projective-space degree \(4d_i^2\), the top possible degree on the \(\mathbb{CP}^{j_i}\)-factor.
The K\"unneth identification used here is justified below in Step 3, using the fact that all cohomology groups involved are torsion-free. Under the K\"unneth decomposition
\[
H^*(X_i\times X_i\times \mathbb{CP}^{j_i};\mathbb Z)
\cong
H^*(X_i;\mathbb Z)\otimes H^*(X_i;\mathbb Z)\otimes H^*(\mathbb{CP}^{j_i};\mathbb Z),
\]
the distinguished term
\[
\pi_1^*\Delta_{d_i}(c(q_i-p_i))\,
\pi_2^*\Delta_{d_i}(c(q_i-p_i))\,
y_i^{2d_i^2}
\]
belongs to the summand
\[
H^{2d_i^2}(X_i;\mathbb Z)
\otimes
H^{2d_i^2}(X_i;\mathbb Z)
\otimes
H^{4d_i^2}(\mathbb{CP}^{j_i};\mathbb Z).
\]
Every other term arising from the expansion of the two cross factors contains a strictly smaller power of \(u=a-b\), and therefore, after substituting \(u=\rho^*y_i\), has strictly smaller cohomological degree on the \(\mathbb{CP}^{j_i}\)-factor. Hence every such term lies in a different K\"unneth summand from the distinguished term. It follows that no linear combination of the lower-\(y_i\)-degree terms can cancel the distinguished \(y_i^{2d_i^2}\)-term.
Consequently, to prove nonvanishing of
\[
\Delta_{2d_i}(c(q_{i+1}-p_{i+1})),
\]
it is enough to prove that the distinguished product
\[
\pi_1^*\Delta_{d_i}(c(q_i-p_i))
\cdot
\pi_2^*\Delta_{d_i}(c(q_i-p_i))
\cdot
y_i^{2d_i^2}
\]
is nonzero in \(H^*(X_{i+1};\mathbb Z)\). This is what we verify next.
\medskip

\noindent
{\bf Step 3: Why the distinguished product is nonzero.}
By the induction hypothesis,
\[
\Delta_{d_i}(c(q_i-p_i))\neq 0
\quad\text{in}\quad
H^{2d_i^2}(X_i;\mathbb Z).
\]
We first claim that its pullbacks along the two coordinate projections
\[
\pi_1,\pi_2:X_i\times X_i\times \mathbb{CP}^{j_i}\longrightarrow X_i
\]
remain nonzero.
To see this, fix a basepoint
\[
(x_0,\ell_0)\in X_i\times \mathbb{CP}^{j_i},
\]
and consider the inclusion
\[
\iota_1:X_i\longrightarrow X_i\times X_i\times \mathbb{CP}^{j_i},
\qquad
\iota_1(x)=(x,x_0,\ell_0).
\]
Then
\[
\pi_1\circ \iota_1=\operatorname{id}_{X_i}.
\]
Consequently,
\[
\iota_1^*\pi_1^*
=
(\pi_1\circ \iota_1)^*
=
\operatorname{id}_{H^*(X_i;\mathbb Z)}.
\]
Thus \(\pi_1^*\) is injective on cohomology. In particular,
\[
\pi_1^*\Delta_{d_i}(c(q_i-p_i))\neq 0.
\]
Exactly the same argument, applied to the inclusion
\[
\iota_2:X_i\longrightarrow X_i\times X_i\times \mathbb{CP}^{j_i},
\qquad
\iota_2(x)=(x_0,x,\ell_0),
\]
shows that
\[
\pi_2^*\Delta_{d_i}(c(q_i-p_i))\neq 0.
\]
Next, by construction,
\[
y_i
=
c_1(\gamma_{i,1}\otimes \gamma_{i,2}^{-1})
=
\pi_{\mathbb{CP}^{j_i}}^*
c_1(\xi_{i,1}\otimes \xi_{i,2}^{-1}).
\]
The line bundles \(\xi_{i,1}\) and \(\xi_{i,2}\) were chosen so that
\[
c_1(\xi_{i,1}\otimes \xi_{i,2}^{-1})^{j_i}
\neq 0
\quad\text{in}\quad
H^{2j_i}(\mathbb{CP}^{j_i};\mathbb Z).
\]
Since \(j_i=2d_i^2\), this gives
\[
y_i^{2d_i^2}\neq 0
\quad\text{in}\quad
H^{4d_i^2}(X_{i+1};\mathbb Z).
\]
We now use the K\"unneth formula. The cohomology groups of all spaces appearing in
the construction are torsion-free. Indeed, Grassmannians have free abelian integral
cohomology with Schubert classes as a basis, and complex projective spaces also have
torsion-free integral cohomology; see Milnor--Stasheff~\cite[Chapter 14]{MilnorStasheffCharacteristicClasses} or Fulton~\cite[Chapter 9]{FultonYoungTableaux}. Since the spaces \(X_i\) are obtained from these by finite products, the
K\"unneth theorem gives
\[
H^*(X_i\times X_i\times \mathbb{CP}^{j_i};\mathbb Z)
\cong
H^*(X_i;\mathbb Z)
\otimes
H^*(X_i;\mathbb Z)
\otimes
H^*(\mathbb{CP}^{j_i};\mathbb Z)
\]
(see \cite[Chapter 3]{BottTuDifferentialForms}).  Under this identification, the product
\[
\pi_1^*\Delta_{d_i}(c(q_i-p_i))
\cdot
\pi_2^*\Delta_{d_i}(c(q_i-p_i))
\cdot
y_i^{2d_i^2}
\]
corresponds to the pure tensor
\[
\Delta_{d_i}(c(q_i-p_i))
\otimes
\Delta_{d_i}(c(q_i-p_i))
\otimes
c_1(\xi_{i,1}\otimes \xi_{i,2}^{-1})^{2d_i^2}.
\]
Each tensor factor is nonzero. Since the tensor product is taken over free abelian
groups, this pure tensor is nonzero. Therefore
\[
\pi_1^*\Delta_{d_i}(c(q_i-p_i))
\cdot
\pi_2^*\Delta_{d_i}(c(q_i-p_i))
\cdot
y_i^{2d_i^2}
\neq 0.
\]
Applying the injective map \(\rho^*\), we also have
\[
\rho^*\pi_1^*\Delta_{d_i}(c(q_i-p_i))
\cdot
\rho^*\pi_2^*\Delta_{d_i}(c(q_i-p_i))
\cdot
\rho^*(y_i^{2d_i^2})
\neq 0.
\]
Thus the distinguished maximal-degree term identified in Step 2 is nonzero.
Since this distinguished term occurs with nonzero coefficient and lies in the
maximal possible degree on the \(\mathbb{CP}^{j_i}\)-factor, it cannot be cancelled
by any term involving a lower power of \(y_i\). Hence
\[
\rho^*\Delta_{2d_i}(c(q_{i+1}-p_{i+1}))\neq 0.
\]
Finally, because \(\rho^*\) is injective, we conclude that
\[
\Delta_{2d_i}(c(q_{i+1}-p_{i+1}))\neq 0
\]
in
\[
H^{2(2d_i)^2}(X_{i+1};\mathbb Z).
\]
This proves the lemma.
\end{proof}
Combining this lemma with the initial Grassmannian computation in Subsection~\ref{grassman}, it
follows by induction that
\[
\Delta_{d_i}(c(q_i-p_i))\neq 0
\quad\text{for every } i.
\]
By the Thom--Porteous obstruction, each pair \((p_i,q_i)\) is therefore totally
degeneracy-forcing. Equivalently, every bundle map $p_i\rightarrow q_i$
vanishes somewhere on \(X_i\).

\section{The uniform tracial completion obstruction}

In this section we obstruct uniform property $\Gamma$ for the inductive limit
algebra \(A\) of Section~\ref{nonsimpleconstruction}.  First we note that it is immediate from the definition of uniform property $\Gamma$ that $A$ has uniform property $\Gamma$ if and only if $\mathrm{M}_n(A)$ does for each $n \in \mathbb{N}$.
If a unital separable
nuclear \(C^*\)-algebra with no finite-dimensional representations has uniform
property \(\Gamma\), then its uniform tracial completion, with its designated
traces, has complemented partitions of unity \cite[Theorem 4.6]{CETW:IMRN}.  For type \(\mathrm{II}_1\)
factorial tracially complete \(C^*\)-algebras with complemented partitions of
unity, Evington--Tikuisis prove real rank zero, stable rank one, and the
corresponding Cuntz-semigroup description; in particular, projections in matrix
amplifications are compared by their values on the designated traces
\cite{CETW:IMRN,CCEGSTW,ET:rrsr}. Thus, in order to certify the failure of uniform property
\(\Gamma\) for \(A\), it suffices to construct, for some \(n\), two projections in
\[
        M_n(A^u)\cong (M_n(A))^u
\]
which agree on all normalized matrix-amplified designated traces
\(\tau^{(n)}\), \(\tau\in T(A)\), but are nevertheless not Murray--von Neumann
equivalent in \(M_n(A^u)\).

The projections \(P\) and \(Q\) below form such a pair. Their equality
on traces is forced by the constant-rank equality
\[
        \operatorname{rank}(p_i)=\operatorname{rank}(q_i)
\]
at every finite stage. Their failure to be equivalent is due to the surviving
topological content of the construction: a partial isometry between
them in the uniform tracial completion can be approximated at a finite stage,
where it yields a bundle map \(p_i\to q_i\) which is nonzero on every fiber.
This contradicts the total degeneracy-forcing property established in Section
\ref{schubert}, namely that every bundle map \(p_i\to q_i\) vanishes somewhere.
The only additional input needed is trace-theoretic and automatic for the
non-simple limit constructed in Section~\ref{nonsimpleconstruction}: every
normalized extreme trace at a finite stage extends to a trace on the
inductive limit.

\subsection{Equal trace and trace survival}
Throughout this section, let
\[
A=\lim_{\longrightarrow}(A_i,\phi_i)
\]
be the inductive limit constructed in Section \ref{nonsimpleconstruction}, and import the notation from that section.
The compatible sequences of projections $(p_i)$ and $(q_i)$ define projections
\[
P\in M_k(A^u),\qquad Q\in A^u.
\]
When comparing $P$ and $Q$ relative to $M_k(A^u)$, we regard $Q$ as the diagonal
projection
\[
Q\oplus 0_{k-1}\in M_k(A^u).
\]
For notational simplicity, we continue to denote this matrix-amplified copy by
$Q$ when no confusion can arise.

\begin{lemma}
\label{lem:nonsimple-point-fiber-traces}
Let \(A=\lim (A_i,\phi_i)\) be the inductive limit constructed in
Section~\ref{nonsimpleconstruction}. For every stage \(i\) and every point
\(x\in X_i\), the normalized point-fiber trace on
\[
A_i=q_i(C(X_i)\otimes\mathcal K)q_i
\]
at \(x\) extends to a trace \(\tau_x\in T(A)\). Equivalently, when
\(\tau_x|_{A_i}\) is represented by a probability measure on \(X_i\), that
measure is the point mass \(\delta_x\).
\end{lemma}

\begin{proof}
Fix \(x_i=x\in X_i\). Choose arbitrary points
\[
\ell_r\in \mathbb{CP}^{j_r},
\qquad r\geq i,
\]
and define points \(x_r\in X_r\), for \(r\geq i\), recursively by
\[
x_{r+1}=(x_r,x_r,\ell_r)
\in X_r\times X_r\times \mathbb{CP}^{j_r}=X_{r+1}.
\]
Let \(\operatorname{tr}_{x_r}\) denote the normalized fiber trace on \(A_r\) at
\(x_r\). If \(a\in A_r\), then
\[
\phi_r(a)=
(\pi_1^*a\otimes \gamma_{r,1})
\oplus
(\pi_2^*a\otimes \gamma_{r,2}).
\]
Evaluating at \(x_{r+1}=(x_r,x_r,\ell_r)\) gives
\[
\phi_r(a)(x_{r+1})
=
\bigl(a(x_r)\otimes \gamma_{r,1}(\ell_r)\bigr)
\oplus
\bigl(a(x_r)\otimes \gamma_{r,2}(\ell_r)\bigr).
\]
The fibers \(\gamma_{r,1}(\ell_r)\) and \(\gamma_{r,2}(\ell_r)\) are
one-dimensional, and the two summands have the same rank. Hence tensoring by
these line fibers does not change the normalized trace, and
\[
\operatorname{tr}_{x_{r+1}}(\phi_r(a))
=
\operatorname{tr}_{x_r}(a).
\]
Thus the point-fiber traces \((\operatorname{tr}_{x_r})_{r\geq i}\) are compatible
with the connecting maps. They define a trace \(\tau_x\in T(A)\) whose restriction
to \(A_i\) is the normalized fiber trace at \(x\). The representing measure of
this restriction is therefore \(\delta_x\).
\end{proof}

It is immediate that
\[
\tau_i^{(k)}(p_i)=\tau_i^{(k)}(q_i\oplus 0_{k-1})
\]
since $p_i$ and $q_i$ have equal rank.  Passage to the inductive limit gives \begin{equation}\label{equaltrace}
\tau^{(k)}(P)=\tau^{(k)}(Q), \ \forall \tau \in T(A).
\end{equation}
\subsection{Non-equivalence in the uniform tracial completion}
We now prove the main obstruction result. This is the argument which ultimately
uses the Thom--Porteous classes; the trace-theoretic input is precisely
Lemma~\ref{lem:nonsimple-point-fiber-traces}.

\begin{theorem}\label{thm:nonsimple-nonequivalence}
Let \(A=\lim (A_i,\phi_i)\) be the non-simple inductive limit
constructed in Section~\ref{nonsimpleconstruction}. Then $P\not\sim Q$
relative to $M_k(A^u)$,
where $Q$ is regarded as $Q\oplus 0_{k-1}$.
\end{theorem}
\begin{proof}
Suppose, toward a contradiction, that
$P\sim Q$
inside $M_k(A^u)$. This may require increasing $k$ in general, but this is inconsequential as it only has to be done once.  Then there exists a partial isometry
$V\in M_k(A^u)$ such that
\[
V^*V=P,\qquad VV^*=Q.
\]
We shall approximate \(V\) by finite-stage contractions.  By the definition
of the uniform tracial completion, the unit ball of \(M_k(A)\) is uniformly
\(2\)-norm dense in the unit ball of \(M_k(A^u)\).  After identifying each
\(M_k(A_i)\) with its image in \(M_k(A)\), the unit ball of the algebraic
inductive-limit union
\[
   \bigcup_i M_k(A_i)
\]
is norm dense in the unit ball of \(M_k(A)\).  Hence there are contractions
\[
   x_m\in M_k(A_{i(m)})
\]
such that
\[
   \|x_m-V\|_{2,u}\longrightarrow 0.
\]
At that stage, $P$ and $Q$ are represented by
\[
p_{i(m)}\in M_k(A_{i(m)}),
\qquad
q_{i(m)}\oplus 0_{k-1}\in M_k(A_{i(m)}).
\]
For notational simplicity, we write $q_{i(m)}$ for this matrix-amplified copy
when forming corners in $M_k(A_{i(m)})$.
Set
\[
y_m=q_{i(m)}x_m p_{i(m)}
\]
so that $y_m:p_{i(m)}\rightarrow q_{i(m)}$ is a bundle map.
Since $QVP=V$, we have
\[
\|y_m-V\|_{2,u}
=
\|q_{i(m)}x_m p_{i(m)}-QVP\|_{2,u}
\leq
\|x_m-V\|_{2,u}
\longrightarrow 0.
\]
Hence
\[
y_m^*y_m\longrightarrow V^*V=P
\]
in uniform $2$-norm. Indeed, since $y_m$ and $V$ are contractions,
\[
\|y_m^*y_m-V^*V\|_{2,u}
\leq
\|y_m^*(y_m-V)\|_{2,u}
+
\|(y_m^*-V^*)V\|_{2,u}
\leq
2\|y_m-V\|_{2,u}
\longrightarrow 0.
\]
Choose $\varepsilon>0$ so small that $k\varepsilon^2<1$.
For $m$ sufficiently large, write
\[
i=i(m),\qquad y=y_m,\qquad p=p_i,\qquad q=q_i.
\]
Then
\begin{equation}\label{largemap}
\|y^*y-p\|_{2,u}<\varepsilon.
\end{equation}
Since $y=qyp$
and $y$ is a contraction, we have
\[
0\leq y(x)^*y(x)\leq p(x), \ \forall x \in X_i.
\]
We now convert the uniform $2$-norm estimate (\ref{largemap}) into a pointwise fiber estimate.
Fix
$x\in X_i$.
By Lemma~\ref{lem:nonsimple-point-fiber-traces}, the normalized point-fiber trace
at \(x\) extends to a trace
$\tau_x\in T(A)$.
It then follows from (\ref{largemap}) that
\[
\operatorname{tr}_x
\left(
(p(x)-y(x)^*y(x))^2
\right)
<
\varepsilon^2
\]
for every \(x\in X_i\).  Set
\[
r(x)=\operatorname{rank}(y(x)).
\]
Since
\[
\operatorname{rank}(p(x))=d_i,
\]
we have
\[
\dim\ker y(x)=d_i-r(x).
\]
The positive operator
\[
p(x)-y(x)^*y(x)
\]
acts as the identity on
\[
\ker y(x)\subseteq p(x)\mathbb C^M.
\]
Therefore
\[
\operatorname{Tr}
\left(
(p(x)-y(x)^*y(x))^2
\right)
\geq d_i-r(x).
\]
The fiber unit of $M_k(A_i)$ has rank $kd_i$, so after normalized trace we obtain
\[
\operatorname{tr}_x
\left(
(p(x)-y(x)^*y(x))^2
\right)
\geq
\frac{d_i-r(x)}{kd_i}.
\]
Combining this with the previous estimate gives
\[
\frac{d_i-r(x)}{kd_i}
<
\varepsilon^2.
\]
It follows that
\[r(x)>
\left(
1-k\varepsilon^2
\right)d_i.
\]
Since $k\varepsilon^2<1$, we have $r(x)>0$ and so
\[
y(x)\neq 0, \ \forall x \in X_i.
\]
In other words, the bundle map $y:p_i\rightarrow q_i$
is nowhere zero. This contradicts the total degeneracy-forcing property established in
Section~\ref{schubert}. Therefore $P\not\sim Q$
inside $M_k(A^u)$.
\end{proof}

\subsection{Failure of uniform property \texorpdfstring{$\Gamma$}{Gamma}}
We now give the general conclusion.
\begin{theorem}\label{thm:nonsimple-no-uniform-gamma}
Let \(A=\lim (A_i,\phi_i)\) be the non-simple inductive limit
constructed in Section~\ref{nonsimpleconstruction}. Then \(A\) does not have
uniform property \(\Gamma\).
\end{theorem}

\begin{proof}
Suppose, toward a contradiction, that \(A\) has uniform property \(\Gamma\) and let
$A^u$ be the uniform tracial completion of \(A\) with respect to \(T(A)\). The algebra
\(A\) is the unital separable nuclear stably finite inductive limit constructed
in Section~\ref{nonsimpleconstruction}, and it has no finite-dimensional
representations by construction. It follows that
projections in matrix amplifications of \(A^u\) are compared by their values on
the designated traces coming from \(T(A)\) \cite{CETW:IMRN, CCEGSTW}.
We have
\[
\tau^{(k)}(P)=\tau^{(k)}(Q), \ \forall \tau\in T(A),
\]
by construction, so comparison by
designated traces gives $P\sim Q$
inside $M_k(A^u)$. This contradicts Theorem~\ref{thm:nonsimple-nonequivalence}, so $A$ does not have uniform property $\Gamma$.
\end{proof}

\section{A simple example}\label{sec:simple-example}

We now pass from the non-simple inductive limit of Section~\ref{nonsimpleconstruction}
to a simple AH algebra.  The point-evaluation part of the construction is the
standard Villadsen simplicity device: one inserts point-evaluation summands
according to a countable schedule so that every nonzero positive element is made
full at some later stage, while the normalized rank of the inserted summands is
kept summable by telescoping the topological part of the system
\cite{Villadsen:JFA,Villadsen:JAMS}.  We record the precise form of this device
needed here, together with the short survival estimates for the old
Thom--Porteous branch.

Fix once and for all a number \(0<\eta<1\).  We freely replace the original
system \((A_i,\phi_i)\) by a telescoping, relabelled with the same notation.  The
forced-degeneracy conclusion of Section~\ref{schubert} is unaffected by this
operation.

\begin{proposition}\label{prop:sparse-simplefication-survival}
There is an inductive system
\[
   B=\lim_{\longrightarrow}(B_i,\psi_i),
   \qquad
   B_i=s_i(C(X_i)\otimes\mathcal K)s_i,
\]
with \(q_i\leq s_i\), satisfying the following properties.

\begin{enumerate}
\item The limit \(B\) is simple, separable, unital, nuclear, and non-elementary.

\item Each connecting map decomposes as
\[
   \psi_i=\psi_i^{\mathrm{top}}\oplus\psi_i^{\mathrm{pt}},
\]
where \(\psi_i^{\mathrm{pt}}\) is a finite direct sum of point evaluations and
\(\psi_i^{\mathrm{top}}\) is the old diagonal topological map, extended from the
\(q_i\)-corner to the larger \(s_i\)-corner.  In particular,
\[
   \psi_i^{\mathrm{top}}|_{A_i}=\phi_i:A_i\longrightarrow A_{i+1}.
\]
If
\[
   \alpha_i=
   \frac{\operatorname{rank}(s_{i+1}^{\mathrm{top}})}
        {\operatorname{rank}(s_{i+1})},
\]
then
\[
   \delta:=\prod_{i=1}^{\infty}\alpha_i>1-\eta.
\]

\item For every \(m\) and every \(x\in X_m\), there is a trace
\(\tau_x\in T(B)\) whose restriction to \(B_m\) is represented by a probability
measure \(\mu_x\) on \(X_m\) satisfying
\[
   \mu_x(\{x\})\geq \delta.
\]

\item For every \(m\),
\[
   \frac{\operatorname{rank}(q_m)}{\operatorname{rank}(s_m)}\geq \delta.
\]
Consequently, if \(\operatorname{tr}_{k,s_m}\) denotes the normalized fiber trace
on \(M_k(B_m)\), then
\[
   \operatorname{tr}_{k,s_m}(p_m(x))
   =
   \frac{\operatorname{rank}(p_m)}{k\operatorname{rank}(s_m)}
   \geq \frac{\delta}{k},
   \qquad x\in X_m.
\]

\item Let \(\kappa_m^B:B_m\to B^u\) denote the canonical map into the uniform
tracial completion, and use the same notation for its matrix amplifications.  Put
\[
   \widehat q_m=q_m\oplus 0_{k-1}\in M_k(B_m).
\]
Then the sequences
\[
   \bigl(\kappa_m^B(p_m)\bigr)_m,
   \qquad
   \bigl(\kappa_m^B(\widehat q_m)\bigr)_m
\]
are Cauchy in the uniform \(2\)-norm of \(M_k(B^u)\).  Let
\[
   P_B=\lim_{m\to\infty}\kappa_m^B(p_m),
   \qquad
   Q_B=\lim_{m\to\infty}\kappa_m^B(\widehat q_m).
\]
Then \(P_B\) and \(Q_B\) are projections in \(M_k(B^u)\), and
\[
   \tau^{(k)}(P_B)=\tau^{(k)}(Q_B),
   \qquad \tau\in T(B).
\]
\end{enumerate}
\end{proposition}

\begin{proof}
We first recall the standard point-evaluation construction.  Enumerate the pairs
\((m,U)\), where \(m\geq 1\) and \(U\) ranges over a countable basis of nonempty
open subsets of \(X_m\).  At the stages prescribed by this enumeration, add a
finite number of point-evaluation summands chosen so that, for each pair
\((m,U)\), some later composite \(B_m\to B_j\) contains a direct summand
\[
   b\longmapsto b(x),
   \qquad x\in U.
\]
This is the standard scheduling argument used in Villadsen's AH constructions
\cite{Villadsen:JFA,Villadsen:JAMS}. Since the old topological maps may be
telescoped to have arbitrarily large multiplicity, the point-evaluation summands
can be inserted with normalized rank smaller than any prescribed summable
sequence.  Indeed, choose a summable sequence
\((\varepsilon_i)\) with \(0<\varepsilon_i<1\) and
\[
\prod_i(1-\varepsilon_i)>1-\eta .
\]
At the \(i\)-th step of the schedule, suppose that \(t_i\) point-evaluation
summands are required.  Before inserting them, telescope the old topological
system far enough that the topological part from the current stage to the next
has multiplicity \(M_i\) with
\[
\frac{t_i}{M_i+t_i}<\varepsilon_i .
\]
After relabelling the terminal algebra as \(B_{i+1}\), define
\(\psi_i^{\mathrm{top}}\) to be this telescoped topological diagonal map applied
to the whole \(s_i\)-corner, and define \(\psi_i^{\mathrm{pt}}\) to be the direct
sum of the required point evaluations, placed in orthogonal summands of the new
unit.  The restriction of \(\psi_i^{\mathrm{top}}\) to the old \(q_i\)-corner is
exactly the corresponding old Thom--Porteous connecting map.  Moreover,
\[
\alpha_i
=
\frac{\operatorname{rank}(s_{i+1}^{\mathrm{top}})}
     {\operatorname{rank}(s_{i+1})}
=
\frac{M_i}{M_i+t_i}
>
1-\varepsilon_i .
\]
Consequently
\[
\prod_i\alpha_i
\geq
\prod_i(1-\varepsilon_i)
>
1-\eta .
\]
Thus, we may arrange
\[
   \sum_i(1-\alpha_i)<\infty,
   \qquad
   \prod_i\alpha_i>1-\eta.
\]
The scheduling condition implies simplicity: if \(0\neq b\in (B_m)_+\), then the
open set \(\{x\in X_m:b(x)\neq 0\}\) contains some basic open set \(U\); at a
later stage the image of \(b\) has a nonzero constant point-evaluation summand
and is therefore full in the corresponding homogeneous block.  Hence every
nonzero ideal in the limit contains the unit.  The limit is non-elementary
because the first building block has infinite spectrum and the connecting maps
are injective.  The remaining assertions in (1) and (2) are immediate from the
construction.

We next prove point-mass survival.  Fix \(m\) and \(x\in X_m\).  Follow the
diagonal topological branch beginning at \(x\): choose points \(z_r\in X_r\),
\(r\geq m\), with \(z_m=x\), so that every topological eigenvalue map from stage
\(r+1\) to stage \(r\) sends \(z_{r+1}\) to \(z_r\).  For \(s>m\), the normalized
fiber trace on \(B_s\) at \(z_s\), restricted to \(B_m\), is represented by a
probability measure \(\mu_{m,s}\) on \(X_m\).  All branches which follow only
topological summands from stages \(m\) through \(s-1\) send \(z_s\) to \(x\), and
their total normalized rank proportion is
\[
   \prod_{r=m}^{s-1}\alpha_r.
\]
Therefore
\[
   \mu_{m,s}(\{x\})
   \geq
   \prod_{r=m}^{s-1}\alpha_r
   \geq
   \prod_{r=1}^{\infty}\alpha_r
   =\delta.
\]
Passing to a subnet of these finite-stage traces gives a trace \(\tau_x\in T(B)\).
If \(\mu_x\) is the measure representing \(\tau_x|_{B_m}\), then the Portmanteau
theorem applied to the closed set \(\{x\}\) gives
\[
   \mu_x(\{x\})\geq \delta.
\]
This proves (3).

For the rank estimate, note that \(s_1=q_1\).  Passing from stage \(m\) to stage
\(m+1\), the old projection \(q_m\) follows only the topological part, while the
unit \(s_m\) follows both the topological and point-evaluation parts.  Hence
\[
   \frac{\operatorname{rank}(q_{m+1})}{\operatorname{rank}(s_{m+1})}
   =
   \alpha_m
   \frac{\operatorname{rank}(q_m)}{\operatorname{rank}(s_m)}.
\]
Iteration gives
\[
   \frac{\operatorname{rank}(q_m)}{\operatorname{rank}(s_m)}
   =
   \prod_{r=1}^{m-1}\alpha_r
   \geq
   \prod_{r=1}^{\infty}\alpha_r
   =\delta.
\]
Since \(\operatorname{rank}(p_m)=\operatorname{rank}(q_m)\), the displayed
fiber-trace estimate follows.  This proves (4).

It remains to prove (5).  For \(m<n\), let \(e_{n,m}\in B_n\) be the projection
onto the direct summands of the composite \(\psi_{n,m}:B_m\to B_n\) which follow
only topological summands from stages \(m\) through \(n-1\).  Put
\[
   E_{n,m}=1_k\otimes e_{n,m}\in M_k(B_n).
\]
On this summand the restriction to the old \(q_m\)-corner agrees exactly with
the old Thom--Porteous connecting map.  Therefore
\[
   E_{n,m}\psi_{n,m}^{(k)}(p_m)E_{n,m}=p_n,
   \qquad
   E_{n,m}\psi_{n,m}^{(k)}(\widehat q_m)E_{n,m}=\widehat q_n.
\]
The complementary normalized rank is unchanged by matrix amplification: for
every \(\sigma\in T(B_n)\),
\[
   \sigma^{(k)}(1-E_{n,m})
   =
   1-\prod_{r=m}^{n-1}\alpha_r
   \leq
   \sum_{r=m}^{n-1}(1-\alpha_r).
\]
Since \(E_{n,m}\) reduces both
\(\psi_{n,m}^{(k)}(p_m)\) and \(\psi_{n,m}^{(k)}(\widehat q_m)\), the differences
\[
   \psi_{n,m}^{(k)}(p_m)-p_n,
   \qquad
   \psi_{n,m}^{(k)}(\widehat q_m)-\widehat q_n
\]
are supported on \(1-E_{n,m}\) and have norm at most one.  Hence, for every
\(\sigma\in T(B_n)\),
\[
   \bigl\|\psi_{n,m}^{(k)}(p_m)-p_n\bigr\|_{2,\sigma^{(k)}}^2
   \leq
   \sigma^{(k)}(1-E_{n,m})
   \leq
   \sum_{r=m}^{n-1}(1-\alpha_r),
\]
and therefore
\[
   \bigl\|\psi_{n,m}^{(k)}(p_m)-p_n\bigr\|_{2,u,B_n}
   \leq
   \left(\sum_{r=m}^{n-1}(1-\alpha_r)\right)^{1/2},
\]
and similarly for \(\widehat q_m\).  Since \(\sum_i(1-\alpha_i)<\infty\), both
sequences are Cauchy in \(M_k(B^u)\).  Their limits are projections because
multiplication is uniformly \(2\)-norm continuous on bounded sets.  Finally,
\(p_m\) and \(\widehat q_m\) have the same constant fiber rank in \(M_k(B_m)\) for
every \(m\), so every trace on \(B_m\), and hence every restriction of a trace on
\(B\), takes the same value on them.  Passing to the uniform \(2\)-norm limit gives
\[
   \tau^{(k)}(P_B)=\tau^{(k)}(Q_B),
   \qquad \tau\in T(B).
\]
\end{proof}

\begin{theorem}[The simple point-evaluation limit has no uniform property \(\Gamma\)]
\label{thm:simple-point-evaluation-no-gamma}
Let \(B\) be the simple AH algebra obtained in
Proposition~\ref{prop:sparse-simplefication-survival}.  Then \(B\) does not have
uniform property \(\Gamma\).  In particular, there is a simple, separable,
unital, nuclear AH algebra without uniform property \(\Gamma\).
\end{theorem}

\begin{proof}
By Proposition~\ref{prop:sparse-simplefication-survival}, the algebra \(B\) is a
simple, separable, unital, nuclear, non-elementary AH algebra.  It remains only
to prove the failure of uniform property \(\Gamma\).

We first show that
\[
   P_B\not\sim Q_B
   \qquad\text{inside } M_k(B^u).
\]
Suppose, toward a contradiction, that there is a partial isometry
\(v\in M_k(B^u)\) such that
\[
   v^*v=P_B,
   \qquad
   vv^*=Q_B.
\]
Set
\[
   \gamma_0=\frac{\delta}{\sqrt{k}}>0.
\]
Choose \(m\) so large and a contraction \(b\in M_k(B_m)\) so close to \(v\) in
uniform \(2\)-norm that, writing
\[
   \bar p=\kappa_m^B(p_m),
   \qquad
   \bar q=\kappa_m^B(\widehat q_m),
   \qquad
   \bar b=\kappa_m^B(b),
\]
we have
\[
   \|\bar p-P_B\|_{2,u,B}<\frac{\gamma_0}{12},
   \quad
   \|\bar q-Q_B\|_{2,u,B}<\frac{\gamma_0}{12},
   \quad
   \|\bar b-v\|_{2,u,B}<\frac{\gamma_0}{12}.
\]
This is possible by Proposition~\ref{prop:sparse-simplefication-survival} and by
Kaplansky density for the uniform tracial completion \cite{CCEGSTW}.  Put
\[
   c=\widehat q_m b p_m\in \widehat q_m M_k(B_m)p_m.
\]
Since \(v=Q_BvP_B\), and all elements involved are contractions,
\[
\begin{aligned}
   \|\kappa_m^B(c)-v\|_{2,u,B}
   &=
   \|\bar q\bar b\bar p-Q_BvP_B\|_{2,u,B} \\
   &\leq
   \|(\bar q-Q_B)\bar b\bar p\|_{2,u,B}
   +\|Q_B(\bar b-v)\bar p\|_{2,u,B}
   +\|Q_Bv(\bar p-P_B)\|_{2,u,B} \\
   &<
   \frac{\gamma_0}{4}.
\end{aligned}
\]
For contractions \(x,y\), one has
\[
   \|x^*x-y^*y\|_{2,u}\leq 2\|x-y\|_{2,u}.
\]
Therefore
\[
   \|\kappa_m^B(c^*c)-P_B\|_{2,u,B}<\frac{\gamma_0}{2},
\]
and hence
\[
   \|\kappa_m^B(c^*c-p_m)\|_{2,u,B}<\gamma_0.
\]

The element \(c\) is fiberwise a bundle map from the old obstruction bundle
\(p_m\) to the old bundle \(q_m\).  By the total degeneracy-forcing property
proved in Section~\ref{schubert}, there is a point \(x\in X_m\) such that
\[
   c(x)=0.
\]
Since \(c=\widehat q_m b p_m\) and \(\|b\|\leq 1\), we have \(0\leq c^*c\leq p_m\).
Thus \(p_m-c^*c\geq 0\), and at the point \(x\),
\[
   (p_m-c^*c)(x)=p_m(x).
\]
Choose the trace \(\tau_x\in T(B)\) supplied by
Proposition~\ref{prop:sparse-simplefication-survival}.  If \(\mu_x\) represents
\(\tau_x|_{B_m}\), then \(\mu_x(\{x\})\geq\delta\), and
Proposition~\ref{prop:sparse-simplefication-survival} also gives
\[
   \operatorname{tr}_{k,s_m}(p_m(x))\geq\frac{\delta}{k}.
\]
Consequently
\[
\begin{aligned}
   \|\kappa_m^B(p_m-c^*c)\|_{2,u,B}^2
   &\geq
   \tau_x^{(k)}\bigl(\kappa_m^B((p_m-c^*c)^2)\bigr) \\
   &=
   \int_{X_m}
   \operatorname{tr}_{k,s_m}\bigl((p_m(y)-c(y)^*c(y))^2\bigr)
   \,d\mu_x(y) \\
   &\geq
   \mu_x(\{x\})\operatorname{tr}_{k,s_m}(p_m(x)) \\
   &\geq
   \delta\cdot\frac{\delta}{k}
   =
   \gamma_0^2.
\end{aligned}
\]
Thus
\[
   \|\kappa_m^B(p_m-c^*c)\|_{2,u,B}\geq\gamma_0,
\]
contradicting the strict inequality above.  Hence \(P_B\not\sim Q_B\) in
\(M_k(B^u)\).

On the other hand, Proposition~\ref{prop:sparse-simplefication-survival} gives
\[
   \tau^{(k)}(P_B)=\tau^{(k)}(Q_B),
   \qquad \tau\in T(B).
\]
If \(B\) had uniform property \(\Gamma\), then comparison of projections by
designated traces in matrix amplifications of \(B^u\) would give
\[
   P_B\sim Q_B
   \qquad\text{inside }M_k(B^u)
\]
\cite{CETW:IMRN,CCEGSTW,ET:rrsr}.  This contradicts the non-equivalence just
proved.  Therefore \(B\) does not have uniform property \(\Gamma\).
\end{proof}

Finally, the simple example remains genuinely quadratic in the displayed AH
presentation.  Indeed, for the original topological branch we proved
\[
\frac{\dim(X_m)}{\operatorname{rank}(q_m)^2}=2
\]
for every \(m\).  Proposition~\ref{prop:sparse-simplefication-survival} gives
\[
\delta
\leq
\frac{\operatorname{rank}(q_m)}{\operatorname{rank}(s_m)}
\leq
1.
\]
Therefore
\[
2\delta^2
\leq
\frac{\dim(X_m)}{\operatorname{rank}(s_m)^2}
=
\frac{\dim(X_m)}{\operatorname{rank}(q_m)^2}
\left(
\frac{\operatorname{rank}(q_m)}{\operatorname{rank}(s_m)}
\right)^2
\leq
2
\]
and the simple AH presentation given above has quadratic dimension
growth.

\bibliography{references}
\end{document}